\documentclass[12pt]{amsart}
\usepackage{amsmath,amsthm,mathscinet}
\usepackage{tabularx}
\usepackage{hyperref}
\newtheorem{thm}{Theorem}[section]
\newtheorem{lem}[thm]{Lemma}

\theoremstyle{definition}
 
\theoremstyle{remark} 
\newtheorem{rem}[thm]{Remark}

\let\abs=\envert

\let\Lm=\Lambda
\let\vph=\varphi
\newcommand{\floor}[1]{\left\lfloor#1\right\rfloor}

\newcommand{\acr}{\newline\indent}

\begin{document}

\title[An analog of perfect numbers]{An analog of perfect numbers involving the unitary totient function}
\author[Tomohiro Yamada]{Tomohiro Yamada*}
\address{\llap{*\,}Center for Japanese language and culture\acr
                   Osaka University\acr
                   562-8558\acr
                   8-1-1, Aomatanihigashi, Minoo, Osaka\acr
                   JAPAN}
\email{tyamada1093@gmail.com}


\begin{abstract}
We shall give some results for an integer divisible by its unitary totient.
\end{abstract}

\maketitle

\section{Introduction}\label{intro}

In 1932, Lehmer \cite{Leh1} asked whether there exists any composite numbers $n$ such that 
the Euler totient $\vph(N)$ divides $N-1$ or not.
Obviously, if $N$ is prime, then $\vph(N)=N-1$ and vice versa.
Lehmer's problem is still unsolved.

On the other hand, it is easy to show that $\vph(N)$ divides $N$
if and only if $N$ is of the form $2^m 3^n$ (see, for example, p.p. 196--197 of \cite{Sie}).
Such an integer can be considered as a totient version of a multiperfect number,
i.e., an integer dividing its sum of divisors.
However, it is not known whether there exists any odd multiperfect numbers,
nor whether there exist infinitely many even multiperfect numbers.

An divisor $d$ of $N$ is called a unitary divisors of $N$ if $\gcd(d, N/d)=1$,
denoted by $d\mid\mid N$.
Subbarao \cite{Sub} considered the problem analogous to Lehmer's problem
involving $\vph^*$, the unitary analogue of $\vph$.
So $\vph^*$ is defined by
\begin{equation}
\vph^*(N)=\prod_{p^e\mid\mid N}(p^e-1),
\end{equation}
where the product is over all prime powers unitarily dividing $N$.
We call the value $\vph^*(N)$ the {\it unitary totient} of an integer $N$.
Subbarao conjectured that $\vph^*(N)$ divides $N-1$ if and only if $N$ is a prime power.
This conjecture is still unsolved.
However, Subbarao and Siva Rama Prasad \cite{SS}
showed that $N$ must have at least eleven distinct prime factors if $N$ is not a prime power
and $\vph^*(N)$ divides $N-1$.

Now the problem naturally arises when $\vph^*(N)$ divides $N$.
We can easily find some instances (see also \url{https://oeis.org/A319481}):
\begin{equation}\label{eq01}
\begin{split}
N_1 & =1, N_2=2, N_3=2\cdot 3, N_4=2^2\cdot 3,\\
N_5 & =2^3\cdot 3\cdot 7,\\
N_6 & =2^4\cdot 3\cdot 5,\\
N_7 & =2^5\cdot 3\cdot 5\cdot 31,\\
N_8 & =2^8\cdot 3\cdot 5\cdot 17,\\
N_9 & =2^{11}\cdot 3\cdot 5\cdot 11^2\cdot 23\cdot 89,\\
N_{10} & =2^{16}\cdot 3\cdot 5\cdot 17\cdot 257,\\
N_{11} & =2^{17}\cdot 3\cdot 5\cdot 17\cdot 257\cdot 131071,\\
N_{12} & =2^{32}\cdot 3\cdot 5\cdot 17\cdot 257\cdot 65537.
\end{split}
\end{equation}
All of these examples satisfy $N=2\vph^*(N)$ except $N=1=\vph^*(1)$ and $N=6=3\vph^*(6)$.

Such integers were implicitly referred in \cite{FJ} as possible orders of groups with
perfect order subsets, groups $G$ with the number of elements in each order subset dividing $\abs{G}$.
Integers equal to twice of its unitary totient had been introduced in the OEIS in 1999
by Yasutoshi Kohmoto\\ \url{https://oeis.org/A030163}.

At the JANT meetings held at Kyoto in the spring of 2009, the author showed some basic results
for integers with this property.
Here the author showed that $N$ must be divisible by one of the first five (i.e. all known)
Fermat primes and have a prime power factor $>10^8$.
Later Ford, Konyagin and Florian Luca \cite{FKL} studied properties of such integers
and showed that such integers $N$ must be divisible by $3$ and $N/\vph^*(N)\leq 85$.

This paper is a revised version of the unpublished note presented at the JANT meetings
and contains some new results which the author recently proved.

We would like to state main results.

\begin{thm}\label{thm2}
If $N$ is an integer divisible by $\vph^*(N)$ other than given in \eqref{eq01},
then $\omega(N)\geq 8$, where $\omega(N)$ denotes the number of distinct prime factors of $N$.
\end{thm}

\begin{thm}\label{thm3}
If $N$ is none of the twelve instances given above, then $N$ must be divisible by an odd prime factor
at least $10^5$ and an odd prime power at least $10^8$.
\end{thm}

\begin{thm}\label{thm5}
There exist only finitely many integers $N$ divisible by $\vph^*(N)$
which are products of consecutive primes.
\end{thm}

Theorems \ref{thm2} and \ref{thm3} of our results are proved by elementary means
but require a fair amount of computation.
Theorem \ref{thm5} are proved using sieve methods combined with
the analytic theory of the distribution of primes in arithmetic progressions.

\section{Basic properties}
From now on, we assume that $N$ is an integer satisfying $\varphi^*(N)\mid N$.

We set $h(n)=n/\vph^*(n)$ for a positive integer $n$.
Hence, $N$ is $\vph^*$-multiperfect if and only if $h(N)$ is an integer.
We factor $N=2^e p_1^{e_1} p_2^{e_2} \cdots p_r^{e_r}$ and put $P_i=p_i^{e_i}$,
where $p_1<p_2<\cdots<p_r$ are the odd prime factors of $N$.
We write $M_k=2^e P_1 P_2 \cdots P_k$ for $k=1, 2, \ldots, r$ and $M_0=2^e$.
Moreover, let $Q_1<Q_2<\cdots<Q_m$ be the odd prime power factors of $N$.
It is easy to see that $Q_k\geq p_k$ for every $k$.
Indeed, among $Q_1, Q_2, \ldots, Q_k$ is some $Q_j=p_t^{e_t}$ with $t\geq k$.
We note that $P_i$ and $Q_i$ are not necessarily equal.
For example, if $N=2^e\cdot 3^2\cdot 5\cdot 11\cdots$, then $P_1=3^2$ and $P_2=5$
while $Q_1=5$ and $Q_2=3^2$.

\begin{lem}\label{lm11}
We have the following properties:
\begin{itemize}
\item[a)] If $N=2^e$, then $e=1$.
\item[b)] For each $i=1, 2, \ldots, r$, $p_i^{e_i}-1$ divides $N$.
In particular, $p_i-1$ must divide $N$.
\item[c)] If $m$ is a squarefree integer dividing $N$, then $\vph(m)=\vph^*(m)$ must divide $N$.
\item[d)] If $2^e\mid\mid N$, then $N$ has at most $e$ distinct odd prime factors.
More exactly, if $2^f$ divides $h(N)$ and $2^{f_i}\mid (p_i-1)$ for each odd prime factor
$p_i (i=1, 2, \ldots, r)$ of $N$, then we have $f+\sum_{i=1}^r f_i\leq e$.
\end{itemize}
\end{lem}

\begin{proof}
\begin{itemize}
\item[a)] $\varphi^*(N)=(2^e-1)\mid 2^e=N$ and therefore we must have $2^e=2$.
\item[b)] $(p_i-1)\mid(p_i^{e_i}-1)\mid\varphi^*(N)\mid N$.
\item[c)] Factor $m=\prod_j q_j$ and set $f_j$ to be the exponent of $q_j$ dividing $N$.
Since $\vph^*(q_j)=(q_j-1)\mid (q_j^{f_j}-1)=\vph^*(q_j^{f_j})$ for each $j$,
we have $\vph^*(m)\mid \vph^*(\prod_j q_j^{f_j})\mid \vph^*(N)\mid N$.
\item[d)]
If $2^f$ divides $h(N)$ and $2^{f_i}\mid (p_i-1)$ for each odd prime factor
$p_i (i=1, 2, \ldots, m)$ of $N$, then 
$2^{\sum_{i=1}^r f_i}\mid (p_1-1)\cdots(p_r-1)\mid \vph^*(N)$
and therefore $2^{f+\sum_{i=1}^r f_i}\mid h(N)\vph^*(N)=N$.
\end{itemize}
\end{proof}

It immediately follows from d) that $1$ is the only odd integer
divisible by its unitary totient.

The property b) makes no reference to the exponent of $p$ dividing $N$.
This is the key tool of our study.
To express the latter part of the property b) in other words,
if $p$ does not divide $N$, then no prime $\equiv 1\pmod{p}$ divides $N$.
In particular, the smallest odd prime factor of $N$ must be a Fermat prime,
a prime of the form $2^{2^k}+1$.

The property c) is a generalization of the latter part of the property b).
We write $\prod_{i=1}^k q_i^{f_i} \rightarrow \prod_{j=1}^l r_j^{g_j}$
if $\prod_{j=1}^l r_j^{g_j}$ divides $\prod_{i=1}^k (q_i^{f_i}-1)$
for two sets of distinct primes $q_i$'s and $r_j$'s.
Now the property c) can be generalized further.
\begin{lem}\label{lm12}
Let $m_i (i=1, 2, \ldots, k)$ be squarefree integers,
$p, q$ be primes and $f\geq 0, g>0$ be arbitrary integers.
If $q^{gl}\mid\mid N$ for some integer $l>0$ and
$q^g\rightarrow m_1\rightarrow m_2\rightarrow \cdots \rightarrow m_k\rightarrow p^f$,
then $p^f\mid N$.
\end{lem}
\begin{proof}
We see that $m_1\mid (q^g-1)\mid (q^{gl}-1)=\vph^*(q^{gl})\mid \vph^*(N)\mid N$.
Since $m_1$ is squarefree, we have $m_2\mid \vph^*(m_1)\mid \vph^*(N)\mid N$
by the property c).
Inductively, each $m_i (i=1, 2, \ldots , k)$ divides $N$.
Hence, $p^f\mid\vph^*(m_k)\mid \vph^*(N)\mid N$.
\end{proof}

Now we introduce two well-known lemmas concerning prime factors
of the $n$-th cyclotomic polynomial, which we denote by $\Phi_n(X)$.
Lemma \ref{lm21} follows from Theorems 94 and 95 in Nagell \cite{Nag}.
Lemma \ref{lm22} has been proved by Bang \cite{Ban} and rediscovered by many authors
such as Zsigmondy \cite{Zsi}, Dickson \cite{Dic} and Kanold \cite{Ka1, Ka2}.

\begin{lem}\label{lm21}
Let $p, q$ be distinct primes with $q\neq 2$ and $n$ be a positive integer.
Then, $q$ divides $p^c-1$ if and only if the multiplicative order of $q$ modulo $p$
divides $n$.
\end{lem}

\begin{lem}\label{lm22}
If $a$ is an integer greater than $1$, then $a^n-1$ has
a prime factor which does not divide $a^k-1$ for any $k<n$,
unless $(a, n)=(2, 1), (2, 6)$ or $n=2$ and $a+1$ is a power of $2$.
Moreover, such a prime factor must be congruent to $1$ modulo $n$.
\end{lem}

Since $2^6-1=3^2\cdot 7$, we see that $a^k-1$ always has a prime factor congruent to
$1$ modulo $k$ unless $(a, k)=(2, 1)$.  Thus the following result holds.

\begin{lem}\label{lm13}
If $a$ is prime, $k$ is a positive integer and $a^k\mid\mid N$, then $k\mid N$.
\end{lem}

\begin{proof}
We may assume that $k>1$.
Since $a$ is prime, $a^k-1=\vph^*(a^k)$ divides $N$ by b) of Lemma \ref{lm11}.
By Lemma \ref{lm22}, $a^k-1$ has at least one prime factor $p\equiv 1\pmod{k}$,
which divides $N$.
Using b) of Lemma \ref{lm11} again, we see that $k$ divides $N$.
\end{proof}

Concerning the number of distinct prime factors of $N$,
it is not difficult to see the following result.

\begin{lem}\label{lma1}
If $N\neq 1, 2, 6, 12, 168, 240$,
then $2^5\mid N$ and $\omega(N)\geq 5$.
\end{lem}
\begin{proof}
If $\omega(N)=1$, then $N=2^e$.
The only possibility is $e=1$ since, if $e>1$, then $h(N)<2/\vph^*(2)=2$.

If $\omega(N)=2$, then $N=2^{e_0} p_1^{e_1}$.
If $h(N)\geq 3$, then we must have $e_0=1$ and $p_1^{e_1}=3$
since, otherwise, $h(N)<2\times (3/2)=3$.
If $h(N)=2$, then $e_0\geq 2$ by Lemma \ref{lm11}.
Now the only possibility is $e_0=2$ and $p_1^{e_1}=3$
since, otherwise, $h(N)<(4/3)\times (3/2)=2$.

If $N=2^2 p_1^{e_1} p_2^{e_2}$,
then $p_1\equiv p_2\equiv 3\pmod{4}$ and $h(N)$ must be odd by Lemma \ref{lm11},
implying that $h(N)\geq 3$.
But $h(N)\leq (4/3)(3/2)(7/6)=7/3<3$, which is a contradiction.

If $N=2^3 p_1^{e_1} p_2^{e_2} p_3^{e_3}$,
then $h(N)\geq 3$ and $p_1\equiv p_2\equiv p_3\equiv 3\pmod{4}$ by Lemma \ref{lm11}.
But $h(N)\leq (8/7)(3/2)(7/6)(11/10)=11/5<3$.

If $N=2^4 p_1^{e_1} p_2^{e_2} p_3^{e_3} p_4^{e_4}$,
then we must have $h(N)\geq 3$ and $p_i\equiv 3\pmod{4} (i=1, 2, 3, 4)$ by Lemma \ref{lm11}
but $h(N)\leq (16/15)(3/2)(7/6)(11/10)(19/18)<3$.

If $N=2^3 p_1^{e_1} p_2^{e_2}$,
then we must have $h(N)=2$ since $h(N)\leq (8/7)(3/2)(5/4)=15/7<3$.
Hence, $p_1\equiv p_2\equiv 3\pmod{4}$.
The only possibility is $(p_1^{e_1}, p_2^{e_2})=(3, 7)$
since, otherwise, $h(N)<(8/7)(3/2)(7/6)=2$.

If $N=2^e p_1^{e_1} p_2^{e_2}$ with $e\geq 4$,
then we must have $(e, p_1^{e_1}, p_2^{e_2})=(4, 3, 5)$
since, otherwise, $h(N)<(16/15)(3/2)(5/4)=2$.

If $N=2^4 p_1^{e_1} p_2^{e_2} p_3^{e_3}$,
then $h(N)=2$ since $h(N)\leq (16/15)(3/2)(5/4)(7/6)=7/3<3$.
Now we must have $p_1\equiv p_2\equiv p_3\equiv 3\pmod{4}$ by Lemma \ref{lm11}.
we must have $(p_1^{e_1}, p_2^{e_2})=(3, 7)$
and therefore $p_3^{e_3}=15$, which is inadmissible.

Now, we see that, unless $N=1, 2, 6, 12, 168, 240$, we must have $2^5\mid N$ and $\omega(N)\geq 4$.
\end{proof}

In general, there exist only finitely many positive integers $N$ divisible by $\vph^*(N)$
with a given number of distinct prime factors.

\begin{lem}\label{lm23}
$N\leq 2^{2^{r+1}}-2^{2^r}$.
\end{lem}
\begin{proof}
It is implicit in Lemma 2 of \cite{Coo} that, for any positive integers $a, b, k$ and $1<n_1<n_2<\cdots <n_k$
satisfying
\begin{equation}
\prod_{i=1}^k\left(1-\frac{1}{n_i}\right)\leq\frac{a}{b}<\prod_{i=1}^{k-1}\left(1-\frac{1}{n_i}\right),
\end{equation}
the inequality
\begin{equation}\label{eq20}
a\prod_{i=1}^k n_i\leq (a+1)^{2^k}-(a+1)^{2^{k-1}}
\end{equation}
must hold.
An explicit proof is given by \cite{Nie} with \eqref{eq20} replaced by $a\prod_{i=1}^k n_i<(a+1)^{2^k}$.
A slight change enables to prove \eqref{eq20}.
Indeed, observing that $an_1\leq a(a+1)=(a+1)^2-(a+1)$ for $k=1$, we can prove \eqref{eq20}
by induction of $r$ proceeding as the proof of Lemma 1 of \cite{Nie}.
Now the lemma immediately follows applying this with $a=1, b=N/\vph^*(N), k=r+1$ and $n_1=2^e, n_{i+1}=P_i$.
\end{proof}

This gives the following upper bound for the $P_i$'s.

\begin{lem}\label{lma0}
For any $i=1, 2, \ldots, r$, we have $P_i(P_i-1)\mid N$
and $P_i\leq 2^{2^{r-1}}$.
\end{lem}
\begin{proof}
The former statement immediately follows from b. of Lemma \ref{lm11}.
By Lemma \ref{lm23}, we have $P_i(P_i-1)\leq N\leq 2^{2^{r+1}}-2^{2^r}$,
which immediately gives the latter statement.
\end{proof}

Moreover, we can see that, for a given $M_k$, the next prime factor $p_{k+1}$ must be bounded.

\begin{lem}\label{lm29}
For each $k=0, 2, \ldots, r-1$, we have
\begin{equation}\label{eq291}
\left(\frac{p_{k+1}}{p_{k+1}-1}\right)^{r-k}\geq \prod_{i=k+1}^r \frac{p_i}{p_i-1}\geq \frac{h(N)}{h(M_k)}.
\end{equation}
Furthermore, $p_{k+1}-1$ must divide $M_k/\gcd(M_k, \vph^*(M_k))$.
\end{lem}
\begin{proof}
\eqref{eq291} immediately follows by observing that
\begin{equation}
\prod_{i=k+1}^m \frac{p_i}{p_i-1}\geq h(p_{k+1}^{e_{k+1}}\cdots p_r^{e_r})=\frac{h(N)}{h(M_k)}.
\end{equation}
We can easily see that $L_k=\gcd(M_k, \vph^*(M_k))(p_{k+1}-1)$ divides $N$
and any prime factor of $L_k$ must be smaller than $p_{k+1}$.
Hence, $L_k$ divides $M_k$ and the latter statement holds.
\end{proof}

\section{The number of distinct prime factors of $N$}

In Lemma \ref{lma1}, we already proved that $\omega(N)\geq 5$.
In this section, we shall prove Theorem \ref{thm2}.

Assume that $N\neq 1, 2, 6, 12, 168, 240$ is an integer such that $\omega(N)\leq 7$ and $\vph^*(N)$ divides $N$.
If $h(N)\geq 3$, then we must have $\omega(N)\geq 8$ since, otherwise,
\begin{equation}
h(N)<h(2^5\cdot 3\cdot 5\cdot 7\cdot 11\cdot 13\cdot 17)=\frac{32}{31}\cdot\frac{3}{2}\cdot\frac{5}{4}\cdot\frac{7}{6}\cdot\frac{11}{10}\cdot\frac{13}{12}\cdot\frac{17}{16}<3.
\end{equation}
Hence, in the remaining part of this section, we may assume that,
$N\neq 1, 2, 6, 12, 168, 240$ is an integer such that $\omega(N)\leq 7$ and $N=2\vph^*(N)$.

We would like to show that $N$ must be one of the instances given in \eqref{eq01}.
This would be proved via several lemmas.

\begin{lem}\label{lmb1}
$3\mid N$.
\end{lem}
\begin{proof}
Although this is proved in \cite{FKL}, we would like to give a proof
since the proof is much simpler given that $\omega(N)\leq 7$.
Assume that $3$ does not divide $N$.
By Lemma \ref{lma1}, we must have $\omega(N)\geq 5$.
If $\omega(N)=5$, then $h(N)\leq h(2^5\cdot 5\cdot 7\cdot 11\cdot 13\cdot 17)<2$
and, if $\omega(N)\geq 6$, then $h(N)\leq h(2^6\cdot 5\cdot 7\cdot 11\cdot 13\cdot 17\cdot 19)<2$.
Hence, we have a contradiction and therefore $3$ must divide $N$.
\end{proof}

\begin{lem}\label{lmb2}
$3\mid\mid N$.
\end{lem}
\begin{proof}
From the previous lemma, we know that $3\mid N$.
Assume that $3^2\mid N$.
Observing that
$h(2^7\cdot 3^2\cdot 5\cdot 7\cdot 17\cdot 19\cdot 23)<2$,
we must have $P_2=5, P_3=7$ and $P_4\leq 13$.

If $P_4=11$, then $P_5=13$ or $17$ by Lemma \ref{lm29}.
If $(P_4, P_5)=(11, 13)$, then $2^{11}\mid N$ by Lemma \ref{lm11}
and therefore $2047/1028\leq h(3^2\cdot 5\cdot 7\cdot 11\cdot 13\cdot P_6)<2$.
Hence, we must have $43<P_6<47$, which is impossible.
If $(P_4, P_5)=(11, 17)$, then $2^{13}\mid N$ and $23<P_6<29$, a contradiction
(we cannot have $P_6=25=p_2^2$ or $27 =p_1^3$).

If $P_4=13$, then $P_5=17, 2^{14}\mid N$ and $P_6<19$, a contradiction.
Hence, $3\mid\mid N$.
\end{proof}

\begin{lem}\label{lmb3}
$5\mid N$.
\end{lem}
\begin{proof}
We know that $P_1=3$ by the previous lemma.
If $p_2>7$, then $h(N)\leq h(2^7\cdot 3\cdot 13\cdot 17\cdot 23\cdot 29\cdot 47)<2$, which is a contradiction.

Now we assume that $p_2=7$.
We observe that $\gcd(p_i-1, 15)=1$ for $i\geq 3$, which implies that neither $11$ nor $13$ divide $N$.
Hence, we see that $p_3\geq 17, p_4\geq 29, p_5\geq 59$ and $p_6\geq 113$.
If $p_3=17$, then $h(N)\leq h(2^9\cdot 3\cdot 7\cdot 17\cdot 29\cdot 59\cdot 113)<2$.
If $p_3>17$, then $p_3\geq 29, p_4\geq 59, p_5\geq 113$ and
$h(N)\leq h(2^6\cdot 3\cdot 5\cdot 29\cdot 59\cdot 113\cdot 127)<2$.
Thus we see that $p_2=7$ is impossible.
Hence, we must have $p_2=5$.
\end{proof}

\begin{lem}\label{lmb4}
$5\mid\mid N$.
\end{lem}
\begin{proof}
We know that $p_1=3, e_1=1$ and $p_2=5$.
We shall show that $e_2$ cannot be divisible by any prime $\leq 29$ unless $e_2\geq 9$ is a power of three.

{\it Case I.} If $e_2$ is even, then $N$ has no prime factor $\equiv 1\pmod{3}$
since otherwise $3^2\mid (5^2-1)(p_i-1)\mid \vph^*(N)\mid N$ for some $i$, contrary to $3\mid\mid N$.
Moreover, $2^9\mid 2(3-1)(5^2-1)(p_3-1)(p_4-1)(p_5-1)(p_6-1)\mid N$.
Hence, we must have
$h(N)\leq h(2^9\cdot 3\cdot 5^2\cdot 11\cdot 17\cdot 23\cdot 29)<2$.

{\it Case II.} If $e_2=3$, then Lemma \ref{lm29} gives that $p_3\leq 11$
since $2^8$ must divide $N$ and $h(2^8\cdot 3\cdot 5^3\cdot 13\cdot 17\cdot 19\cdot 23)<2$.

If $p_3=7$, then we must have $p_4, p_5, p_6 \equiv 2\pmod{3}$ as above.
Since $h(2^8\cdot 3\cdot 5^3\cdot 7^2\cdot 11\cdot 17\cdot 23)<2$
and $h(2^8\cdot 3\cdot 5^3\cdot 7\cdot 23\cdot 29\cdot 41)<2$,
we must have $e_3=1$ and $p_4\leq 17$.
In other words, we must have $p_4=11$ or $p_4=17$.

If $p_4=11$, then, we must have $e_4=1$
since $h(2^8\cdot 3\cdot 5^3\cdot 7\cdot 11^2\cdot 17\cdot 23)<2$.
Furthermore, by Lemma \ref{lm29}, $p_5\leq 73$ and $p_5-1=2^m u$ with $u$ dividing $5^3\cdot 7\cdot 11$.
Hence, $p_5=17, 23, 29, 41$ or $71$.
If $e_5\geq 3$, then $p_5^{e_5}-1$ has always a prime factor $\geq 47$.  Hence, $p_6\geq 47$ and $h(N)<2$.
If $e_5$ is even, then we must have $3\mid p_5^{e_5}-1$ and $3^2\mid (7-1)(p_5^{e_5}-1)\mid N$ while $3\mid\mid N$,
which is a contradiction.
Hence, we must have $e_5=1$ and, $p_5=41$ or $71$ since otherwise $h(N)\geq h(3\cdot 5^3\cdot 7\cdot 11\cdot 29)>2$.

If $p_5=41$, then Lemma \ref{lm29} gives that $p_6^{e_6}=257$ or $p_6^{e_6}=281$,
leading to a contradiction observing that
$h(2^9\cdot 3\cdot 5^3\cdot 7\cdot 11\cdot 41\cdot 281)>2>h(2^{10}\cdot 3\cdot 5^3\cdot 7\cdot 11\cdot 41\cdot 257)$.
If $p_5=71$, then Lemma \ref{lm29} gives that $p_6^{e_6}=89$ or $p_6^{e_6}=113$,
which is impossible since $h(2^8\cdot 3\cdot 5^3\cdot 7\cdot 11\cdot 71\cdot 89)<2$.
Thus we cannot have $p_4=11$.

If $p_4=17$, then, proceeding as above, we must have $e_4=1$ and $p_5^{e_5}=29$.
It follows from Lemma \ref{lm29} that no prime is appropriate for $p_6$.
Thus we cannot have $p_4=17$ and therefore $p_3=7$.

If $p_3=11$, then $p_4^{e_4}=17$ but Lemma \ref{lm29} yields that
there exists no appropriate prime for $p_5$.

{\it Case III.} If $5$ divides $e_2$, then $11\cdot 71$ divides $N$.
We must have $p_3^{e_3}=7$ since otherwise $h(N)\leq h(2^8\cdot 3\cdot 5^5\cdot 11\cdot 13\cdot 17\cdot 71)<2$.
$p_4^{e_4}=11$.
If $p_5\geq 53$, then $h(N)\leq h(2^8\cdot 3\cdot 5^5\cdot 7\cdot 11\cdot 53\cdot 71)<2$.
If $p_5\leq 47$, then, by Lemma \ref{lm29}, we must have $p_5=17, 23, 29$ or $41$.
If $e_5\geq 2$, then $h(N)\leq h(2^8\cdot 3\cdot 5^5\cdot 7\cdot 11\cdot 17^2\cdot 71)<2$.
If $e_5=1$, then $h(N)>h(\cdot 3\cdot 7\cdot 11\cdot 41\cdot 71)>2$.
Both cases lead to contradictions and therefore $e_2$ cannot be divisible by $5$.

{\it The other cases.} If $7$ divides $e_2$, then $3^2$ must divide $N$ since $5^7\rightarrow 19531\rightarrow 3^2$, which contradicts to Lemma \ref{lmb2}.
Similarly, it follows from
$5^{11}\rightarrow 12207031\rightarrow 521\rightarrow 13$ and $13\cdot 12207031\rightarrow 3^2$,
$5^{13}\rightarrow 305175781\rightarrow 3^2$,
$5^{17}\rightarrow 409\cdot 466344409\rightarrow 3^2$,
$5^{19}\rightarrow 191\rightarrow 19\rightarrow 3$ and $5^{19}\rightarrow 6271\rightarrow 3$
and
$5^{23}\rightarrow 332207361361\rightarrow 3^2$
that, if $11, 13, 17, 19$ or $23$ divides $e_2$, then $3^2$ must divide $N$,
contrary to Lemma \ref{lmb2} again.

Now, if $e_2$ is a power of three, then $3^2\mid e_2$, $19\mid (5^9-1)\mid N$ and $3^2\mid (19-1)\mid N$,
contrary to Lemma \ref{lmb2}.
Thus we see that $e_2$ must have a prime factor $\geq 29$ if $5^2\mid N$.
However, Lemma \ref{lma0} immediately gives that $5^{e_2}\leq 2^{64}$.
Hence, we must have $e_2=1$.
\end{proof}

\begin{lem}\label{lmb5}
If $p=7, 11$ or $13$, then $p^2\mid\mid N$.
If a prime $17\leq p<200$ divides $N$, then $e\leq 2$ and $p\in \{23, 31, 47, 67, 83, 137\}$
or $e=1$ and $p\in \{17, 29, 41, 53, 59, 61, 71, 89, 97, 103, 107, 113, 131, 139, 167, 179, 191, 193, 197\}$.
Moreover, if $p=233, 241, 257, 409, 641, 769, 1021, 1361$ or $1637$ divides $N$, then $p\mid\mid N$.
\end{lem}

\begin{proof}
We shall only give a proof for $p=7$.
Assume that $7\mid N$ and let $e$ be its exponent.
We observe that $7^2$ must divide $N$ since $h(105)>2$.

It follows from $7^3\rightarrow 3^2$, $7^7\rightarrow 4733\rightarrow 3$,
$7^{11}\rightarrow 1123\rightarrow 3$, $7^{17}\rightarrow 2767631689\rightarrow 3^2$
and $7^{19}\rightarrow 419\rightarrow 19\rightarrow 3^2$
together with $7\rightarrow 3$
that, if $e_2$ is divisible by one of $3, 7, 11, 17$ and $19$, then
$3^2$ must divide $N$, contrary to Lemma \ref{lmb2}.

Similarly, if $e_2$ is divisible by $4, 5$ or $13$, then
$5^2$ must divide $N$ since $7^4\rightarrow 5^2$, $7^5\rightarrow 2801\rightarrow 5^2$
and $7^{13}\rightarrow 16148168401\rightarrow 5^2$,
contrary to Lemma \ref{lmb4}.

Since $7^{23}>2^{64}$, Lemma \ref{lma0} yields that we must have $e=2$.
The proof for the other primes is similar to the case $p=7$.
\end{proof}

\begin{lem}\label{lmb6}
$7$ cannot divide $N$.  In other words, $p_3\geq 11$.
\end{lem}
\begin{proof}
Assume that $7$ divides $N$.

We already know that $7^2\mid\mid N$ from the previous lemma.

We must have $11\leq p_4<71$.
since $h(2^{10}\cdot 3\cdot 5\cdot 7^2\cdot 71\cdot 73\cdot 77)<2$.
Hence, $p_4=11, 17, 29$ or $41$.
By the previous lemma, we must have $p_4^{e_4}=11^2, 17, 29$ or $41$.
If $p_4^{e_4}=11^2$, then $3^2\mid (7^2-1)(11^2-1)\mid N$, which contradicts Lemma \ref{lmb2}.
If $p_4^{e_4}=17$, then $h(N)>h(3\cdot 5\cdot 7^2\cdot 17)>2$, a contradiction again.

If $p_4^{e_4}=29$, then, with the aid of Lemma \ref{lm29}, we must have $p_5=41, 59, 71, 101, 113$, $e_5\geq 2$ or $p_5=197, 233$.
By Lemma \ref{lmb5}, we must have $P_5=197$ or $233$.
If $P_5=197$, then Lemma \ref{lm29} yields that $p_6=281$.
However this is impossible since
$h(2^{12}\cdot 3\cdot 5\cdot 7^2\cdot 29\cdot 197\cdot 281)>2>h(2^{13}\cdot 3\cdot 5\cdot 7^2\cdot 29\cdot 197\cdot 281)$.
If $P_5=233$, then there exists no appropriate prime for $p_6$ in view of Lemma \ref{lm29}.

If $p_4^{e_4}=41$, then, combining Lemmas \ref{lm29} and \ref{lmb5}, we must have $P_5=83$.
But, now there exists no appropriate prime for $p_6$ in view of Lemma \ref{lm29}.
\end{proof}

\begin{lem}\label{lmb7}
$N=2^{11}\cdot 3\cdot 5\cdot 11^2\cdot 23\cdot 89$ or $p_3^{e_3}\in \{17, 31, 31^2, 41, 61\}$.
\end{lem}
\begin{proof}
Now we know that $p_3\in \{11, 13, 17, 31, 41, 61\}$.
Assume that $p_3\leq 13$.
Since $h(3\cdot 5\cdot 13)=195/96>2$, we must have $e_3\geq 2$.
By Lemma \ref{lmb5}, we must have $p_3^{e_3}=11^2$.

Now Lemma \ref{lm29} implies that
$p_4=2^{m_4} 11^{f_4}+1$ and $p_4<53$.  The only possibility is $p_4=23$.
Lemma \ref{lmb5} gives that $e_4\leq 2$.
If $e_4=2$, then using Lemma \ref{lm29} again, $p_5<41$ and $p_5=2^{m_5} 11^{f_5} 23^{g_5}+1$, which is impossible.
If $e_4=1$, then Lemma \ref{lm29} gives that $p_5=89$.
If $e=11$, then $N=2^{11}\cdot 3\cdot 5\cdot 11^2\cdot 23\cdot 89$ satisfies $h(N)=2$.
If $e\geq 12$, then we see that $2^{11}<p_6^{e_6}<2^{12}$ and $p_6^{e_6}-1$ divides $2^e\cdot 23\cdot 89$.
However, there exists no such prime $p_6$.
\end{proof}

\begin{lem}\label{lmb8}
We can never have $p_3=31, 41$ or $61$.
\end{lem}

\begin{proof}
If $p_3=31$, then Lemma \ref{lm29} implies that
$p_4<257$ and $p_4$ must be of the form $2^f\cdot 31^g$, which is impossible.
If $P_3=41$, then Lemma \ref{lm29} implies that $p_4<83$ and $p_4$ must be of the form $2^{f_1}\cdot 3^{f_2}\cdot 41^{f_3}$
with $f_2\leq 1$, which is again impossible.
Similarly, $P_3=61$ is impossible.
\end{proof}

Now the remaining case is the case $P_1=3, P_2=5$ and $P_3=17$.

\begin{lem}\label{lmb9}
If $P_1=3, P_2=5, P_3=17, P_4=257$, then
$N=N_{10}, N_{11}$ or $N_{12}$.
\end{lem}
\begin{proof}
If $2^{16}\mid\mid N$, then $N=N_{10} M$ for some integer $M$ coprime to $N_{10}$.
Since $h(N_{10})=2$, we must have $M=1$ and $N=N_{10}$.

If $2^{17}\mid N$, then
$h(2^{e_0}\cdot 3\cdot 5\cdot 17\cdot 257)\leq 2(1-1/131071)$.

If $2^{17}\mid\mid N$, then
$65537\leq p_5\leq Q_5\leq 262143$
and $p_5-1$ must divide $2\cdot 3\cdot 5\cdot 17\cdot 257$.
The only possibility is $p_5=131071$.
If $e_5>1$, then $p_5<p_6\leq Q_5\leq 262143$ and
$p_6-1$ must divide $2\cdot 131071$, which is impossible.
Hence, we have $e_5=1$ and $N=N_{11} M$ for some integer $M$ coprime to $N_{11}$.
Since $h(N_{11})=2$, we must have $M=1$ and $N=N_{11}$.

If $2^{18}\mid N$, then $65537\leq p_5\leq Q_5<174761$
and $p_5-1$ must divide $2^{e-16}\cdot 3\cdot 5\cdot 17\cdot 257$.
Hence, $p_5$ must be $65537, 82241, 87041, 98689, 131071$ or $163841$.
If $p_5=163841$, then $P_6<200000$ and $P_6-1$ divides $2^e\cdot 3\cdot 17\cdot 257$, which is impossible.
If $p_5=131071$, then $P_6<262144$ and $P_6-1$ divides $2^e\cdot 131071$, which is impossible.
If $p_5=98689$, then $2^{24}\mid N$, $P_6<200000$ and $P_6-1$ divides $2^e\cdot 5\cdot 17\cdot 98689$, which is impossible.
If $p_5=87041$, then $2^{27}\mid N$, $P_6<266000$ and $P_6-1$ divides $2^e\cdot 3\cdot 257\cdot 87041$, which is impossible.
If $p_5=82241$, then $2^{23}\mid N$, $P_6<340000$ and $P_6-1$ divides $2^e\cdot 3\cdot 17\cdot 87041$.
Hence, we must have $P_6=328961$, which is impossible since $h(2^{24} P_1 P_2 P_3 P_4 P_5 P_6)>2>h(2^{25} P_1 P_2 P_3 P_4 P_5 P_6)$.

If $p_5=65537$, then $2^{32}\mid N$.
If $p_5=65537$ and $2^{33}\mid N$, then
$2^{32}\leq P_6<2^{33}-1$ and $2^e-1$ must divide $3\cdot 5\cdot 17\cdot 257\cdot 65537\cdot P_6$.
We must have $e=61$ and $P_6=2^{61}-1$.
However, this is inappropriate since $P_6-1$ does not divide $2^{61} \cdot 3\cdot 5\cdot 17\cdot 257\cdot 65537$.
Hence, $2^{33}$ cannot divide $N$ and $N$ must be $N_{12}$.
We note that, if $F_5=2^{32}+1$ were prime, then $e=64$ and $P_6=2^{32}+1$ would suffice.
\end{proof}

\begin{lem}\label{lmb10}
If $P_4\neq 257$, then we must have $P_4\in \{409, 641, 769, 1021\}$.
\end{lem}
\begin{proof}
By Lemma \ref{lm29}, we must have $p_4\in \{31, 41, 61, 97, 103, 137, 193, 241,\\
257, 409, 641, 769, 1021\}$.

If $p_4\leq 257$ and $P_4\neq 257$, then, we must have $e_4>1$ since $h(3\cdot 5\cdot 17\cdot 251)>2$.
It immediately follows from Lemma \ref{lmb5} that
$p_4=31$ or $137$ and $e_4=2$.
In view of Lemma \ref{lm29}, there exists no prime appropriate for $p_5$.

If $p_4>257$, then Lemma \ref{lmb5} immediately gives that $e_4=1$
and $P_4\in \{409, 641, 769, 1021\}$.
\end{proof}

\begin{lem}\label{lmb11}
We cannot have $P_1=3, P_2=5, P_3=17$ and $p_4>257$.
\end{lem}
\begin{proof}
By the previous lemma, $p_4=409, 641, 769$ or $1021$
and, by Lemma \ref{lmb5}, we must have $e_4=1$.
If $p_4\geq 769$, then, by Lemma \ref{lm29}, we must have $p_5\geq 1021$,
which is impossible since $2^{11}$ must divide $N$ and
$h(N)\leq h(2^{11}\cdot 3\cdot 5\cdot 17\cdot 769\cdot 1021\cdot 1031)<2$.
If $p_4=409$, then, $2^{13}$ must divide $N$ and,
by Lemma \ref{lm29}, we must have $409<p_5<1500$
and $p_5-1$ must divide $2^e\cdot 5\cdot 409$.
However, there exists no such prime.
If $p_4=641$, then, by Lemma \ref{lm29}, we must have $p_5=769$.
Observing that $2^{24}$ must divide $N$ and using Lemma \ref{lm29} again,
we find that $p_6<1000$ and $p_6-1$ divides $2^e\cdot 769$,
which is impossible.
Hence, no prime $>257$ is appropriate for $p_4$.
\end{proof}

Now the theorem immediately follows combining Lemmas \ref{lmb7}-\ref{lmb11}.

\section{The largest prime (power) factor of $N$}

In this section, we prove Theorem \ref{thm3}.
We write $P(n)$ for the largest prime divisor of $n$.

We begin by proving the former part of Theorem \ref{thm3}.
Firstly, we need to find all primes $p$ and integers $k$ such that $p<10^5$ and $P(p^k-1)<10^5$.
This work can be done using the method of Goto and Ohno \cite{GO},
who showed that an odd perfect number must have a prime factor $>10^8$.
For the use in the proof of the latter part of Theorem \ref{thm3},
we shall also determine all exponents $k$ for which $P(2^k-1)<10^8$.

\begin{lem}\label{lm51}
If $P(2^k-1)<10^5$, then $k\leq 16$ or $k\in\{18, 20, 21, 22, 24, 25,\\
26, 28, 29, 30, 32, 36, 40, 42, 44, 45, 48, 50, 52, 60, 84\}$.
Moreover, if $P(2^k-1)<10^8$, then $k\leq 30$ or
\[\begin{split}
k\in\{
& 32, 33, 34, 35, 36, 38, 39, 40, 42, 43, 44, 45, 46, 47, 48, 50, 51,\\
& 52, 53, 54, 55, 56, 57, 58, 60, 63, 64, 66, 68, 70, 72, 75, 76, 78,\\
& 81, 84, 90, 92, 96, 100, 102, 105, 108, 110, 132, 140, 156, 180, 210\}.
\end{split}\]
\end{lem}

\begin{rem}
Recently, Stewart \cite{Ste} showed that, given two integers $b>a>1$, we have
$P(a^k-b^k)/k>\exp(\log k/(104\log\log k))$ for $k$ sufficiently large in terms of $\omega(ab)$.
We note that it is easy to show that the largest prime {\it power} divisor of $a^k-b^k$
is $>Ck\vph(k)\log a/\log (k\log a)$ for some absolute constant $C>0$.
\end{rem}

\begin{proof}
Here we shall prove the latter statement.
The latter yields the former after checking each exponent.

We may assume that $k>30$.
Let $\Phi_k(x)=\prod_\zeta (x-\zeta)$ denote the $k$-th cyclotomic polynomial,
where $\zeta$ runs over all primitive $k$-th roots of the unity.
Observing that each primitive $k$-th root of unity appears together with its conjugate
and $(2-\zeta)(2-\bar\zeta)\geq 3$, we have $\Phi_k(2)\geq 3^{\vph(k)/2}$ for any $k$.

By Lemma \ref{lm21}, any prime factor of $\Phi_k(2)$ must be congruent to $1\pmod{k}$.
Moreover, Lehmer\cite{Leh2} shows that if $p<10^8$ and $p^2\mid\Phi_k(2)$, then $p=1093$ or $p=3511$.
Hence, if $P(2^k-1)<10^8$, then the product of primes below $10^8$
congruent to $1\pmod{k}$ must be $\geq 3^{\vph(k)/2}/(1093\cdot 3511)$ and therefore
\begin{equation}
\frac{3^{\vph(k)/2}}{1093\cdot 3511}\leq \prod_{i<10^5/k} (ik+1)<(10^5)^{10^5/k}.
\end{equation}
This yields that
\begin{equation}\label{eq51}
k\varphi(k) \leq 2\cdot 10^8\cdot \log(10^8)/\log 3+\log(1093\cdot 3511).
\end{equation}
If $k\geq 510510$, then $\vph(k)\geq 92160$.  If $144000\leq k<510510$, then $\vph(k)/k\geq 5760/30030$.
In neither case, \eqref{eq51} can hold.

Now we show that $k\leq 500$.  This can be done by computing the product of primes below $10^8$
congruent to $1\pmod{k}$ for each $500<k\leq 144000$.
Indeed, we confirmed that, for each $500<k\leq 144000$,
\begin{equation}
\sum_{p\equiv 1\pmod{k}, p\mid \Phi_k(2)}p<\frac{\vph(k)\log 3}{2}-\log(1093\cdot 3511).
\end{equation}

Finally, for each $k\leq 500$, we factored $2^k-1$ and found all $k$'s satisfying $P(2^k-1)<10^8$,
which are given in the lemma.
\end{proof}

Nextly, we shall determined all prime powers $p^k$ with $3\leq p<10^5$ and $P(p^k-1)<10^5$.

\begin{lem}\label{lm52}
All prime powers $p^k$ such that $3\leq p<10^5$, $k\in \{7, 8, 9, 10, 12,\\ 20\}$ and $P(p^k-1)<10^5$
are given in Table \ref{tbl}.
If $P(3^k-1)<10^5$, then $k\leq 54$.
If $P(5^k-1)<10^5$, then $k\leq 30$.
Furthermore, if $7\leq p<10^5$, $k>10$, $k\neq 12, 20$ and $P(p^k-1)<10^5$, then
$(p, k)=(7, 14), (7, 24), (11, 21), (11, 24),\\ (13, 11), (67, 16)$.
\end{lem}

\begin{table}\label{tbl}
\caption{All prime powers $p^k$ such that $3\leq p<10^5$, $k\in \{7, 8, 9, 10, 12, 20\}$ and $P(p^k-1)<10^5$}\begin{center}
\begin{small}
\begin{tabularx}{\linewidth}{| c | X |}
 \hline
$k$ & $p$ \\
 \hline
$7$ & $2, 3, 5, 7, 11, 19, 59, 79, 269, 359$ \\
$8$ & $3, 5, 7, 11, 13, 17, 19, 31, 37, 41, 43, 47, 59, 67, 79, 83, 107, 127, 137, 149$,
$223, 227, 233, 239, 263, 269, 271, 359, 389, 401, 499, 563, 571, 587, 617, 773$,
$809, 823, 881, 971, 1061, 1091, 1201, 1213, 1319, 1487, 1579, 1637, 1657, 1669$,
$1783, 1907, 2351, 2383, 2399, 2677, 2741, 3109, 3163, 3373, 3631, 3847, 3851$,
$4877, 5167, 6451, 7237, 7699, 8081, 9239, 9397, 9733, 10099, 10181, 10691$,
$11483, 12721, 14051, 14149, 15427, 16067, 16607, 16987, 18979, 19531, 20129$,
$25253, 25633, 27073, 35837, 37783, 41893, 42391, 46327, 46889, 47041, 49253$,
$53831, 57173, 58013, 60101, 62497, 65951, 66541, 69457, 75931, 82241, 82261$,
$84229, 87721, 88339, 88819, 91499, 92333, 95917, 99523$ \\
$9$ & $2, 3, 5, 7, 19, 29, 31, 37, 43, 53, 379, 1019, 63599$ \\
$10$ & $3, 5, 7, 11, 13, 17, 19, 31, 53, 67, 113, 197, 421, 569$ \\
$12$ & $3, 5, 7, 11, 13, 17, 19, 23, 29, 41, 47, 53, 73, 79, 89, 97, 101, 103, 113$,
$137, 139, 197, 251, 271, 307, 367, 389, 397, 401, 421, 467, 479, 487, 907, 1013$,
$1319, 1451, 1627, 1697, 3083, 4027, 22051, 30977, 52889$ \\
$20$ & $3, 5, 7, 13, 17$ \\
 \hline
\end{tabularx}
\end{small}
\end{center}
\end{table}

\begin{proof}
Such prime powers $p^k$ with $k$ \textit{odd prime} can be taken from the table of
Goto and Ohno, who determined all such prime powers with $p<10^8$ and $P(p^k-1)<10^8$.
Their table shows that, if $p, r$ are odd primes such that $r>5, p<10^5$ and $P(p^r-1)<10^5$, then
$(p, r)=(13, 11)$, $p=3$ and $r=11$ or $17$ or
$r=7$ and $p=3, 5, 7, 11, 19, 59, 79, 269$ or $359$.
Furthermore, for any prime $p\geq 3$ and $r_1, r_2\geq 5$, we have $P(p^{r_1 r_2}-1)>10^5$.
We note that, there exist exactly $125$ odd primes $p<10^5$ such that $P(p^5-1)<10^5$.

We confirmed that $P(p^9-1)<10^5$ only for $p=3, 5, 7, 19, 29, 31, 37, 43, 53,\\ 379, 1019$ or $63599$,
among which, $P(p^{27}-1)<10^5$ only for $p=3, 5$ and $P(p^{81}-1)<10^5$ holds for no odd prime $p<10^5$.
Moreover, we confirmed that $P(p^8-1)$ for exactly $116$ odd primes $p<10^5$
and $P(p^{16}-1)$ only for $p=3, 5$ or $67$ among primes below $10^5$.
Moreover, $P(p^{32}-1)>10^5$ for any prime $p<10^5$.

If $p\geq 3, r\geq 7$ and $P(p^{3r}-1)<10^5$, then $(p, r)=(11, 7)$.
If $p$ is an odd prime and $P(p^{15}-1)<10^5$, then $p=3, 5$ and $P(p^{45}-1)>10^5$.
Now we know all prime powers $p^k$ with $k$ odd for which $P(p^k-1)<10^5$.

Finally, we checked each prime $p<10^5$ and exponent $k=2^f r$
with $f\leq 4$ and $r\geq 3$ prime or $r=9, 15, 21, 27$ such that $P(p^r-1)<10^5$.
We found that $p=3, r=14, 18, 24, 30, 54$, $p=5, r=14, 27, 30$, $p=7, r=14, 24$, $p=11, r=21, 24$,
$(p, r)=(13, 11), (17, 24), (67, 16)$ or $(p, r)$ must be in Table \ref{tbl}.
This completes the proof.
\end{proof}

Now our computer search starting from $2^k$ for each $k$ given in Lemma \ref{lm51},
yields that the eleven instances $N_1, N_2, \ldots, N_{10}$ and $N_{12}$
given in the introduction are all ones with $P(N)<10^5$.

Our algorithm can be explained in the following way:
\begin{enumerate}
\item Begin with $p_0=2, e_0=e, i=0, m_i=0$;
\item We factor $p_{m_i}^{e_{m_i}}-1$; For each prime factor $q$ of $p_{m_i}^{e_{m_i}}-1$,
we put $p_j=q$ for some $j$ with $p_j=0$ and relate $p_j$ to $p_{m_i}$ if there is no $j$ with $p_j=q$;
Moreover, we let $f_{m_i, j}$ be the exponent of $p_j$ dividing $p_{m_i}^{e_{m_i}}-1$;
\item For each unmarked $p_j$, we factor $p_j-1$ and mark $p_j$;  For each prime factor $q$ of $p_j-1$,
we put $p_h=q$ for some $h$ with $p_h=0$ and relate $p_h$ to $p_{m_i}$ if there is no $h$ with $p_h=q$;
Moreover, we let $f_{j, h}$ be the exponent of $p_h$ dividing $p_j-1$;
Until all $p_j$'s are marked, we repeat this step;
\item If there is some $j$ with $\sum_h f_{j, h}>e_j>0$, then jump to Step 4A;  Otherwise,
if there is some $j$ with $p_j>0$ and $e_j=0$, jump to Step 4B;  Otherwise jump to Step 4C;
\item[(4A)] We clear all $p_j$'s, all $f_{j, h}$'s for all prime $p_j$ related to $p_{m_i}$;
Jump to Step 5;
\item[(4B)] If there is some $j$ with $e_j=0$, then we set $m_{i+1}$ to be one of such $j$('s);
We set $e_{m_{i+1}}=1$;  We set $i=i+1$;  Jump to Step 2;
\item[(4C)] We output $N=\prod p_j^{e_j}$ where $j$ runs over all $j$ with $p_j\neq 0$;
Jump to Step 5;
\item[(5)] If $p_{m_i}>2$, then jump to Step 5A;  Otherwise, jump to Step 5B;
\item[(5A)] We set $e_{m_i}=e_{m_i}+1$;  If $p_{m_i}^{e_{m_i}}<10^8$, then jump to Step 2;  Otherwise,
We clear $p_{m_i}, e_{m_i}$;  We clear $m_i$;  We set $i=i-1$;  Jump to Step 2;
\item[(5B)] We set $e_{m_i}$ to the next member of the current value of $e_{m_i}$ in the set given in
Lemma \ref{lm52};  If there is no more member, then we terminate;
\end{enumerate}

We illustrate our proof in the case $e=32$.

If $2^{32}\mid\mid N$, then $3\cdot 5\cdot 17\cdot 257\cdot 65537=2^{32}-1$ divides $N$.
In particular, $3$ must divide $N$.
By Lemma \ref{lm52}, we must have $3^k\mid\mid N$ with $1\leq k\leq 12, 14\leq k\leq 18$ or $k\in\{20, 22, 24, 27, 28, 30, 34, 48, 54\}$.

If $k$ is even, then we must have $2^{33}\mid\vph^*(3^k\cdot 5\cdot 17\cdot 257\cdot 65537)\mid N$,
contrary to the assumption that $2^{32}\mid\mid N$.
If $3$ divides $k$, then $13\mid (3^3-1)\mid N$ and
$2^{33}\mid\vph^*(3\cdot 5\cdot 13\cdot 17\cdot 257\cdot 65537)\mid N$, a contradiction again.

If $5$ divides $k$, then $11^2\mid (3^5-1)\mid N$ and $5^2\mid \vph^*(2^{32}\cdot 11)\mid N$.
Hence, we must have $5^l\mid\mid N$ with $2\leq l\leq 10$, $l=12, 14, 15, 16, 18, 20$ or $30$.
If $l$ is even, then $2^{33}\mid\vph^*(3\cdot 5^2\cdot 17\cdot 257\cdot 65537)\mid N$,
a contradiction.  If $3$ divides $l$, then $31\mid (5^3-1)$ divides $N$ and
$2^{33}\mid \vph^*(3\cdot 5\cdot 11\cdot 17\cdot 31\cdot 257\cdot 65537)\mid N$,
a contradiction.  $l=5, 7$ lead to a similar contradiction and
$5$ can never divide $k$.
$p\mid k$ with $p=7, 11, 13, 17$ also lead to a similar contradiction
and we must have $k=1$.

Proceeding similarly, we must have $3\cdot 5\cdot 17\cdot 257\cdot 65537\mid\mid N$
and $N=2^{32}\cdot 5\cdot 17\cdot 257\cdot 65537\mid\mid N$.

Our procedure worked for the other exponents given in Lemma \ref{lm51}
and yielded only twelve instances given in the Introduction.
Some exponents required several hours.
This proves the former part of Theorem \ref{thm3}.

Now we shall prove that $N$ must be divisible by a prime power $>10^8$.
$e$ must belong to the set given in Lemma \ref{lm51}.
As in the previous section, we show that $2^e\mid\mid N$ cannot occur
unless $N$ is one of twelve instances given above
for each $e$ in this set.

For example, we show that $e\neq 210$ in the following way:

$2^{210}-1=3^2\cdot 7^2\cdot 11\cdot 31\cdot 43\cdot 71\cdot 127\cdot 151\cdot 211\cdot 281\cdot 331\cdot 337\cdot 5419\cdot 29191\cdot 86171\cdot 106681\cdot 122921\cdot 152041\cdot 664441\cdot 1564921$ must divide $N$.
Hence, $(3-1)(7-1)(11-1)\ldots(1564921-1)=2^{35}\cdot 3^{20}\cdot 5^{15}\cdot 7^{15}\cdot 11\cdot 23\cdot 43\cdot 113\cdot 127\cdot 139\cdot 181\cdot 439\cdot 1231$ must divide $N$.
Thus $3^{20}>10^8$ must divide $N$, which is contradiction.

Our procedure terminated for all exponents given in Lemma \ref{lm51}
and yielded only twelve instances given in the Introduction.
This completes the proof of Theorem \ref{thm3}.

\section{A product of consecutive primes}

The purpose of this section is to prove the Theorem \ref{thm5}.

In this section, $\sum^\prime$ denotes the sum over primes in a given range.
We need a Brun-Titchmarsh type theorem in the following form.

\begin{lem}\label{lm71}
During the statement and the proof of this lemma, let $c$ denote an effectively computable absolute constant which may take different value at each occurrence.
Moreover, let $Q_1=\log^{9/2} x$ and $Q=x^{1/2}/\log^{9/2} x$.
Then, we have
\begin{equation}
\sideset{}{^\prime}\sum_{q<Q} \max_{y\leq x} \abs{\pi(y; q, 1)-\frac{\pi(y)}{\vph(q)}}<\frac{cx}{\log^2 x}
\end{equation}
for $x>0$.
\end{lem}

\begin{proof}
By Theorem 1.2 of \cite{AH}, we have
\begin{equation}
\begin{split}
& \sum_{q\leq Q}\frac{q}{\vph(q)}\sideset{}{^*}\sum_{\chi\pmod{q}} \max_{y\leq x}\abs{\psi(y, \chi)} \\
& <c_0(4x+2x^\frac{1}{2} Q^2+6x^\frac{2}{3} Q^\frac{3}{2}+5x^\frac{5}{6} Q)\log^\frac{7}{2} x,
\end{split}
\end{equation}
where $c_0=48.83236\cdots$ and $\sum^*_{\chi\pmod{q}}$ denotes
the sum over all primitive characters $\chi\pmod{q}$.

Dividing the sum into intervals of the form $(2^k Q_1, 2^{k+1} Q_1]$
$(k=0, 1, \ldots,\\ \floor{\log(Q/Q_1)/\log 2}$, we have
\begin{equation}
\begin{split}
& \sum_{Q_1<q\leq Q}\frac{1}{\vph(q)}\sideset{}{^*}\sum_{\chi\pmod{q}} \max_{y\leq x}\abs{\psi(y, \chi)} \\
& <c\left(\frac{x}{Q_1}+x^\frac{1}{2} Q+x^\frac{2}{3} Q^\frac{1}{2}+5x^\frac{5}{6}\log x\right)\log^\frac{7}{2} x.
\end{split}
\end{equation}
Limiting $q$ in the sum to primes, we obtain
\begin{equation}
\begin{split}
& \sideset{}{^\prime}\sum_{Q_1<q\leq Q} \max_{y\leq x}\abs{\psi(y; q, 1)-\frac{\psi(y)}{\vph(q)}-1} \\
& <c\left(\frac{x}{Q_1}+x^\frac{1}{2} Q+x^\frac{2}{3} Q^\frac{1}{2}+5x^\frac{5}{6}\log x\right)\log^\frac{7}{2} x
\end{split}
\end{equation}
and, recalling that $Q_1=\log^{9/2} x$ and $Q=x^{1/2}/\log^{9/2} x$,
\begin{equation}\label{eq71}
\sideset{}{^\prime}\sum_{Q_1<q\leq Q} \max_{y\leq x}\abs{\psi(y; q, 1)-\frac{\psi(y)}{\vph(q)}}
<c\frac{x}{\log x}.
\end{equation}

It is implicit in the proof of Theorem 4 of Chen and Wang \cite{ChW} that
for a Dirichlet character $\chi$ modulo $q\leq \log^{9/2} x$ and $x\geq \exp(\exp(9.7))$, we have
\begin{equation}\label{eq72}
\abs{\psi(x, \chi)-E_0 x}\leq \frac{0.022x}{\log^{7.5} x}+E_0 \log x+E_1\frac{x^\beta}{\beta},
\end{equation}
where $E_0=1$ if $\chi$ is principal and $E_0=0$ otherwise
and $\beta$ denotes a real zero of $L(s, \chi)$ greater than $1-0.1077/\log q$ and $E_1=1$ if it exists
and $E_1=0$ otherwise (For more general results, see the author's recent paper \cite{Ymd}).
Moreover, $E_1=1$ occurs for at most one character among all Dirichlet characters modulo $k$.

From Kadiri \cite{Kad}, we know that, there exists at most one modulus $q_0\leq Q_1$
such that a Dirichlet $L$-function $L(s, \chi)$ has a real zero $s=\beta>1-1/4.0904\log Q_1$ for some character $\chi\pmod{q_0}$.
We shall call the modulus $q_0$ exceptional if it exists and other moduli $\leq Q_1$ nonexceptional.
Now, it immediately follows from \eqref{eq72} that, if a prime $q\leq Q_1$ is nonexceptional, then 
\begin{equation}
\max_{y\leq x}\abs{\psi(y; q, 1)-\frac{\psi(y)}{\vph(q)}}<\frac{cx}{\log^{7.5} x}+2x^{1-1/4.0904\log Q_1}<\frac{c_1 x}{\log^{7.5} x}
\end{equation}
for another effectively computable constant $c_1$.
Using Theorem 3 of \cite{LW} stating that $\beta\leq 1-\pi/0.4923q_0^{1/2}\log^2 q_0$,
we have $x^\beta/q_0<cx(\log\log x)^4/\log^2 x$ and therefore
\begin{equation}
\max_{y\leq x}\abs{\psi(y; q_0, 1)-\frac{\psi(y)}{\vph(q_0)}}<\frac{cx}{\log^{7.5} x}+\frac{c^\prime x(\log\log x)^4}{\log^2 x}<\frac{c_2 x}{\log x}
\end{equation}
for some effectively computable constants $c^\prime$ and $c_2$.
Hence, we obtain
\begin{equation}
\sideset{}{^\prime}\sum_{q\leq Q_1} \max_{y\leq x}\abs{\psi(y; q, 1)-\frac{\psi(y)}{\vph(q)}}<\frac{cx}{\log x}.
\end{equation}
Combining with \eqref{eq71}, we have
\begin{equation}\label{eq71a}
\sideset{}{^\prime}\sum_{q<Q} \max_{y\leq x} \abs{\psi(y; q, 1)-\frac{\psi(y)}{\vph(q)}}<\frac{cx}{\log x}.
\end{equation}

Let
\begin{equation}
\Pi(y; q, a)=\sum_{n\leq y, n\equiv a\pmod{q}}\frac{\Lm(n)}{\log n}=\sum_{p^k\leq y, p^k\equiv a\pmod{q}}\frac{1}{k}
\end{equation}
and put $x_1=x/\log^2 x$.
Moreover, we write $E_f(y; q, a)$ for the error term $\abs{f(y; q, a)-f(y)/\vph(q)}$ for
arithmetic functions $f=\pi, \Pi, \theta, \psi$.
Using partial summation, we have
\begin{equation}\label{eq71b}
\begin{split}
E_\Pi(x, q, 1)-& E_\Pi(x_1, q, 1) \\
= & \frac{E_\psi(x, q, 1)}{\log x}-\frac{E_\psi(x_1, q, 1)}{\log x_1}-\int_{x_1}^x \frac{E_\psi(y, q, 1)}{y\log^2 y}dy \\
\leq & \frac{\abs{E_\psi(x, q, 1)}}{\log x}+\frac{\abs{E_\psi(x_1, q, 1)}}{\log x_1}+\frac{\max_{y\leq x}\abs{E_\psi(y, q, 1)}}{\log x_1} \\
\leq & \frac{2\max_{y\leq x}\abs{E_\psi(y, q, 1)}}{\log x_1}+\frac{cx_1}{\log x_1}.
\end{split}
\end{equation}
In Section 8 of \cite{AH}, it is shown that
\begin{equation}\label{eq71c}
\abs{E_\Pi(x_1; q, a)-E_\pi(x_1; q, a)}\leq 2x_1^\frac{1}{2}\leq 2x^\frac{1}{2}
\end{equation}
for $y\leq x$.

Finally, observing that $Q\leq x_1^{1/2}$, the Brun-Titchmarsh inequality given in \cite{MV} gives
\begin{equation}\label{eq71d}
\sideset{}{^\prime}\sum_{q\leq Q}\abs{E_{\pi}(x_1, q, 1)}\leq \sideset{}{^\prime}\sum_{q\leq Q}\frac{4x_1}{q\log x_1}\leq \frac{c_3 x_1\log\log x}{\log x_1}\leq \frac{cx\log\log x}{\log^3 x_1},
\end{equation}
where $c_3$ denotes some effectively computable constant.

Now the Lemma follows from \eqref{eq71a}, \eqref{eq71b}, \eqref{eq71c} and \eqref{eq71d}.
\end{proof}

We shall prove Theorem \ref{thm5}.
Let $0<\alpha_1, \alpha_2<1/2$ be two constants.
Hereafter, $c$ and $c^\prime$ denote some effectively computable constants
which may take different value at each occurrence and depend only on $\alpha_1$ and $\alpha_2$.
Assume that $\vph^*(N)$ divides $N$ and $N$ is divisible by exactly all primes below $x$.
Moreover, we may assume that $x$ is so large that $\abs{\pi(x)-x/\log x}<0.001x/\log x$.
For any prime $q$, $q^{\pi(x; q, 1)}\mid \prod_{p\leq x, p\equiv 1\pmod{q}} (p-1)\mid \vph^*(N)\mid N$.
Hence,
\begin{equation}
\sideset{}{^\prime}\prod_{q\leq x}\left(\frac{q^{\pi(x; q, 1)}}{q^{\pi(x; q, 1)}-1}\right)\geq \prod_{q^f\mid\mid N}\left(\frac{q^f}{q^f-1}\right)=h(N)\geq 2
\end{equation}
and therefore
\begin{equation}\label{eq70}
\begin{split}
\sideset{}{^\prime}\sum_{q\leq x}\frac{1}{q^{\pi(x; q, 1)}}
> & \sideset{}{^\prime}\sum_{q\leq x}\log\left(\frac{q^{\pi(x; q, 1)}}{q^{\pi(x; q, 1)}-1}\right)-\sideset{}{^\prime}\sum_{q\leq x}\frac{1}{q^{2\pi(x; q, 1)}} \\
> & \log 2-\frac{c}{\log x},
\end{split}
\end{equation}
where the last inequality follows observing that $q^{2\pi(x; q, 1)}>c\log^2 x$ for $q<\log x$ by \eqref{eq72}.

Let $U$ be the set of primes $q$ such that $\pi(x; q, 1)<0.001x/(q\log x)$.
We let
\begin{equation}
L=\sum_{p\leq x} \log (p-1)=\sum_{\substack{p, q\leq x,\\ p, q\text{: prime},\\ q\mid (p-1)}} \log q=\sideset{}{^\prime}\sum_{q<x/2} \pi(x; q, 1)\log q
\end{equation}
and divide $L$ into
\begin{equation}
L_1=\sideset{}{^\prime}\sum_{q\leq x^{1/2-\alpha_1}} \pi(x; q, 1)\log q,
\end{equation}
\begin{equation}
L_{2, 1}=\sum_{x^{1/2-\alpha_1}\leq q\leq x^{1-\alpha_2}, q\in U}\pi(x; q, 1)\log q,
\end{equation}
\begin{equation}
L_{2, 2}=\sum_{x^{1/2-\alpha_1}\leq q\leq x^{1-\alpha_2}, q\not\in U}\pi(x; q, 1)\log q
\end{equation}
and
\begin{equation}
L_3=\sideset{}{^\prime}\sum_{x^{1-\alpha_2}\leq q<x/2} \pi(x; q, 1)\log q.
\end{equation}

We can easily see that
\begin{equation}\label{eq731}
L=\sum_{p\leq x} \log (p-1)>\theta(x)-\sum_{p\leq x}\frac{2}{p}>0.999x.
\end{equation}

Since $x^{\alpha_1}>\log^5 x$, we can apply Lemma \ref{lm71} to obtain
\begin{equation}\label{eq732}
\begin{split}
L_1\leq & \pi(x)\sideset{}{^\prime}\sum_{q\leq Q} \frac{\log q}{q-1}+(\alpha_1\log x)\sideset{}{^\prime}\sum_{q\leq Q} \abs{\pi(x; q, 1)-\frac{\pi(x)}{q-1}} \\
\leq & 1.001\pi(x)\log Q+\frac{c\alpha_1 x}{\log x} \\
\leq & 0.501x.
\end{split}
\end{equation}

Using the inequality $\abs{\sum_{q<z}^\prime (\log q)/q-E}<1/2\log z$ with $E=-1.33258\cdots$ for $z\geq 319$ (see Theorem 6 of \cite{RS}), we have
\begin{equation}\label{eq733}
L_{2, 1}\leq \frac{0.001x}{\log x}\sideset{}{^\prime}\sum_{x^{1/2-\alpha_1}\leq q\leq x^{1-\alpha_2}}\frac{\log q}{q}\leq 0.001x
\end{equation}
for sufficiently large $x$.

Let $S(y, a)$ be the number of integers $q\leq y$ such that both $q$ and $aq+1$ are prime.
We observe that
\begin{equation}\label{eq75}
\sideset{}{^\prime}\sum_{x^{1-\alpha_2}\leq q<x/2} \pi(x; q, 1)=\sum_{\substack{x^{1-\alpha_2}\leq q<x/2, q\text{: prime},\\ p=aq+1\leq x, p\text{: prime}}}1
\leq \sum_{2\leq a\leq x^{\alpha_2}}S(x/a, a).
\end{equation}
Proceeding as in Subsection 2.3.2 of \cite{Gre} immediately yields that
\begin{equation}
S(x/a, a)<\frac{c\psi(a)(x/a)}{a\log^2 (x/a)}<\frac{c\psi(a)x}{(1-\alpha_2)^2 a^2 \log^2 x}
\end{equation}
for each $a$, where we note that $c$ does not depend on $a$ and $\psi(a)=a\prod_{p\mid a}(1+1/p)$
now denotes the Dedekind $\psi$-function, not the second Chebyshev function.
Thus the last sum in \eqref{eq75} is at most
\begin{equation}
\frac{cx}{(1-\alpha_2)^2 \log^2 x}\sum_{2\leq a\leq x^{\alpha_2}}\frac{\psi(a)}{a^2}
<\frac{(c+0.0001) c^\prime \alpha_2 x}{(1-\alpha_2)^2 \log x}
\end{equation}
and therefore
\begin{equation}\label{eq734}
L_3<\frac{(c+0.0001)c^\prime \alpha_2 x}{(1-\alpha_2)^2}.
\end{equation}

By \eqref{eq731}, \eqref{eq732}, \eqref{eq733} and \eqref{eq734},
there exists an absolute and effectively computable constant $\delta>0$ such that,
if $\alpha_2<\delta$, then
\begin{equation}\label{eq76}
L_{2, 2}>L-(L_1+L_{2, 1}+L_3)>0.497x
\end{equation}
for sufficiently large $x$.

On the other hand, using the Brun-Titchmarsh theorem again,
\begin{equation}
\begin{split}
L_{2, 2}< & \sum_{\substack{q\not\in U,\\ x^{1/2-\alpha_1}\leq q\leq x^{1-\alpha_2}}}\frac{2x\log q}{(q-1)\log (x/q)} \\
< & \frac{2.0001(1-\alpha_2)x}{\alpha_2}\sum_{\substack{q\not\in U,\\ x^{1/2-\alpha_1}\leq q\leq x^{1-\alpha_2}}}\frac{1}{q}
\end{split}
\end{equation}
and
\begin{equation}\label{eq77}
\begin{split}
\sum_{\substack{q\not\in U,\\ x^{1/2-\alpha_1}\leq q\leq x^{1-\alpha_2}}}\frac{1}{q}>\frac{0.248\alpha_2}{1-\alpha_2}.
\end{split}
\end{equation}

Since $\abs{E_\pi(x; q, 1)}>0.998x/(q\log x)$ for any prime $q<x^{1/2-\alpha_1}$ in $U$,
Lemma \ref{lm71} gives
\begin{equation}
\sum_{q\leq x^{1/2-\alpha_1}, q\in U}\frac{1}{q}<\frac{1.01\log x}{x}\sum_{q\leq x^{1/2-\alpha_1}, q\in U}E_\pi(x; q, 1)<\frac{c}{\log x}.
\end{equation}
Thus \eqref{eq70} gives
\begin{equation}
\begin{split}
& \sum_{\substack{q\in U,\\ x^{1/2-\alpha_1}\leq q\leq x^{1-\alpha_2}}}\frac{1}{q}\geq \sum_{\substack{q\in U,\\ x^{1/2-\alpha_1}\leq q\leq x^{1-\alpha_2}}}\frac{1}{q^{\pi(x; q, 1)}} \\
> & \log 2-\frac{c}{\log x}-\sideset{}{^\prime}\sum_{x^{1-\alpha_2}<q<x} \frac{1}{q}-\sum_{\substack{q\not\in U,\\ q\leq x^{1-\alpha_2}}}\frac{1}{q^{0.001x/(q\log x)}} \\
> & \log 2+\log(1-\alpha_2)-\frac{c^\prime}{\log x}.
\end{split}
\end{equation}
Hence,
\begin{equation}
\begin{split}
\sum_{\substack{q\not\in U,\\ x^{1/2-\alpha_1}\leq q\leq x^{1-\alpha_2}}}\frac{1}{q}
< & \log\frac{1-\alpha_2}{\frac{1}{2}-\alpha_1}+\frac{c}{\log x}-\sum_{\substack{q\in U,\\ x^{1/2-\alpha_1}\leq q\leq x^{1-\alpha_2}}}\frac{1}{q} \\
< & \log\frac{1-\alpha_2}{\frac{1}{2}-\alpha_1}-\log 2-\log (1-\alpha_2)+\frac{c}{\log x} \\
= & \log\frac{1}{1-2\alpha_1}+\frac{c}{\log x}.
\end{split}
\end{equation}
Now \eqref{eq77} implies
\begin{equation}
\frac{0.248\alpha_2}{1-\alpha_2}<\log\frac{1}{1-2\alpha_1}+\frac{c}{\log x}.
\end{equation}

Taking $0<\alpha_1<1$ so that 
\begin{equation}
\frac{(1-\alpha_2)x}{\alpha_2}\log\frac{1}{1-2\alpha_1}<0.248,
\end{equation}
we have $c/\log x>c^\prime$ for sufficiently large $x$, which implies $x<c$.
This proves Theorem \ref{thm5}.

\section{Acknowledgement}

The author is gratefully thankful for the referee's helpful comments and
informations on some literatures including \cite{FKL}.

{}
\vskip 12pt
\end{document}